\documentclass[a4paper]{llncs}

\usepackage{pdfpages}

\usepackage{booktabs}
\usepackage{amssymb}
\usepackage{xcolor}
\usepackage{verbatim} 
\usepackage{multirow}

\usepackage{amsmath}
\usepackage{caption}

\usepackage{graphicx}

\usepackage[english]{babel}
\usepackage{amsfonts}
\usepackage{amsmath}
\usepackage{latexsym}
\usepackage{color}
\usepackage{subfigure}
\usepackage{enumerate}
\usepackage{capt-of}
\usepackage{hyperref}  

\usepackage{ifpdf}

\usepackage[misc]{ifsym}
\usepackage{url}
\newcommand{\keywords}[1]{\par\addvspace\baselineskip
\noindent\keywordname\enspace\ignorespaces#1}

\DeclareMathOperator    \conv           {conv}
\DeclareRobustCommand\sage[1]{\texttt{#1}}

\newcommand{\bb}{\mathbb}

\newcommand{\R}{\bb R}

\newcommand{\floor}[1]{\lfloor#1\rfloor}

\renewcommand{\P}{\mathcal{P}}
\newcommand{\setcond}[2]{\left\{ #1 \,:\, #2 \right\}}

\begin{document}

\mainmatter  

  
\title{Characterization and Approximation of \\ Strong General Dual Feasible Functions}


%
%

\author{Matthias K\"oppe$^{\textrm{\Letter}}$%
\thanks{The authors gratefully acknowledge partial support from the National Science
  Foundation through grant DMS-1320051 (M.~K\"oppe). A part of this work was done while the first author (M.~K\"oppe) was visiting the Simons Institute for the Theory of Computing. It was partially supported by the DIMACS/Simons Collaboration on Bridging Continuous and Discrete Optimization through NSF grant CCF-1740425.}%
\and Jiawei Wang}

%

\institute{Dept.\ of Mathematics, University of California, Davis, USA\\
\email{mkoeppe@math.ucdavis.edu},
\email{jwewang@ucdavis.edu}\\
}

%
%

\graphicspath{{../dff_paper_graphics/}} 
\DeclareGraphicsExtensions{.pdf,.png,.jpg}
\newcommand\Figure[2][\relax]{%
  \begin{figure}[h!]
    \includegraphics[width=.8\textwidth]{#2}
    \caption{\ifx#1\relax#2\else#1\fi}
  \end{figure}
}

\toctitle{Lecture Notes in Computer Science}
\tocauthor{Authors' Instructions}
\maketitle

\begin{abstract}
Dual feasible functions (DFFs) have been used to provide bounds for standard packing problems and valid inequalities for integer optimization problems. In this paper, the connection between general DFFs and a particular family of cut-generating functions is explored. We find the characterization of (restricted/strongly) maximal general DFFs and prove a 2-slope theorem for extreme general DFFs.  We show that any restricted maximal general DFF can be well approximated by an extreme general DFF.
\end{abstract}

\section{Introduction}

Dual feasible functions (DFFs) are a fascinating family of functions $\phi\colon [0,1] \to [0,1]$, which have been used in several combinatorial optimization problems including knapsack type inequalities and proved to generate lower bounds efficiently. DFFs are in the scope of superadditive duality theory,  and superadditive and nondecreasing DFFs can provide valid inequalities for general integer linear programs. Lueker
\cite{Lueker:1983:BPI:1382437.1382833} studied the bin-packing problems and used certain DFFs to obtain lower bounds for the first time. Vanderbeck \cite{vanderbeck2000exact} proposed an exact algorithm for the cutting stock problems which includes adding valid inequalities generated by DFFs. Rietz et al. \cite{rietz2012computing} recently introduced a variant of this theory, in which the domain of DFFs is extended to all real numbers. Rietz et al. \cite{Rietz2014} studied the maximality of the so-called ``general dual feasible functions." They also summarized recent literature on DFFs in the monograph \cite{alves-clautiaux-valerio-rietz-2016:dual-feasible-book}. 
In this paper, we follow the notions in the monograph \cite{alves-clautiaux-valerio-rietz-2016:dual-feasible-book} and study the general DFFs. 

Cut-generating functions play an essential role in generating valid inequalities which cut off the current fractional basic solution in a simplex-based cutting plane procedure. Gomory and Johnson \cite{infinite,infinite2} first studied the corner relaxation of integer linear programs, which is obtained by relaxing the non-negativity of basic variables in the tableau.  Gomory--Johnson cut-generating functions are critical in the superadditive duality theory of integer linear optimization problems, and they have been used in the state-of-art integer program solvers. K\"oppe and Wang \cite{KOPPE2017153} discovered a conversion from minimal Gomory--Johnson cut-generating functions to maximal DFFs.

Y{\i}ld{\i}z and Cornu\'ejols \cite{yildiz2016cut} introduced a generalized model of Gomory--Johnson cut-generating functions. In the single-row Gomory--Johnson model, the basic variables are in $\mathbb{Z}$. Y{\i}ld{\i}z and Cornu\'ejols considered the basic variables to be in any set $S\subset\mathbb{R}$. Their results extended the characterization of minimal Gomory--Johnson cut-generating functions in terms of the generalized symmetry condition. Inspired by the characterization of minimal Y{\i}ld{\i}z--Cornu\'ejols cut-generating functions, we complete the characterization of maximal general DFF.

We connect general DFFs to the classic model studied by Jeroslow \cite{jeroslow1979minimal}, Blair \cite{BLAIR1978147} and Bachem et al. \cite{BACHEM198263} and a relaxation of their model, both of which can be studied in the Y{\i}ld{\i}z--Cornu\'ejols model \cite{yildiz2016cut} with various sets $S$. General DFFs generate valid inequalities for the Y{\i}ld{\i}z--Cornu\'ejols model with $S=(-\infty,0]$, and cut-generating functions generate valid inequalities for the Jeroslow model where $S=\{0\}$. The relation between these two families of functions is explored.

Another focus of this paper is on the extremality of general DFFs. In terms of Gomory--Johnson cut-generating functions, the 2-slope theorem is a famous result of Gomory and Johnson's masterpiece \cite{infinite,infinite2}. Basu et al. \cite{bhm:dense-2-slope} proved that the 2-slope extreme Gomory--Johnson cut-generating functions are dense in the set of continuous minimal functions. We show that any 2-slope maximal general DFF with one slope value $0$ is extreme. This result is a key step in our approximation theorem, which indicates that almost all continuous maximal general DFFs can be approximated by extreme (2-slope) general DFFs as close as we desire. Unlike the 2-slope fill-in procedure Basu et al. \cite{bhm:dense-2-slope} used, we always use $0$ as one slope value in our fill-in procedure, which is necessary since the 2-slope theorem of general DFFs requires $0$ to be one slope value.

 This paper is structured as follows. 
In \autoref{s:LR}, we provide the preliminaries of DFFs from the monograph \cite{alves-clautiaux-valerio-rietz-2016:dual-feasible-book}. The characterizations of maximal, restricted maximal and strongly maximal general DFFs are described in \autoref{s:characterization}. In  \autoref{s:YC}, we explore the relation between general DFFs and a particular family of cut-generating functions in terms of the ``lifting" procedure. The 2-slope theorem for extreme general DFFs is studied in \autoref{s:2slope}. In \autoref{s:approximation}, we introduce our approximation theorem, adapting a parallel construction in Gomory--Johnson's setting  \cite{bhm:dense-2-slope}.

\section{Literature Review}
\label{s:LR}
\begin{definition}[{\cite[Definition 2.1]{alves-clautiaux-valerio-rietz-2016:dual-feasible-book}}]
A function $\phi\colon [0,1] \to [0,1]$ is called a (valid) \emph{classical Dual-Feasible Function} (cDFF), if for any finite index set $I$ of nonnegative real numbers $x_i \in [0,1]$, it holds that,

$$\sum_{i\in I}x_i \le 1 \Rightarrow \sum_{i\in I}\phi(x_i) \le 1 $$

\end{definition}


\begin{definition}[{\cite[Definition 3.1]{alves-clautiaux-valerio-rietz-2016:dual-feasible-book}}]
A function $\phi\colon \mathbb{R} \to\mathbb{R}$ is called a (valid) \emph{general Dual-Feasible Function} (gDFF), if for any finite index set $I$ of real numbers $x_i \in \mathbb{R}$, it holds that,

$$\sum_{i\in I}x_i \le 1 \Rightarrow \sum_{i\in I}\phi(x_i) \le 1 $$

\end{definition}


Despite the large number of DFFs that may be defined, we are only interested in so-called ``maximal" functions since they yield better bounds and stronger valid inequalities.  A cDFF/gDFF is \emph{maximal} if it is not (pointwise) dominated by a distinct cDFF/gDFF.  In order to get strongest valid inequalities, maximality is not enough. A cDFF/gDFF is \emph{extreme} if it cannot be written as a convex combination of other two different cDFFs/gDFFs. In the monograph \cite{alves-clautiaux-valerio-rietz-2016:dual-feasible-book},  the authors explored maximality of both cDFFs and gDFFs.  

\begin{theorem}[{\cite[Theorem 2.1]{alves-clautiaux-valerio-rietz-2016:dual-feasible-book}}]
\label{thm:maximality-classical}
A function $\phi\colon [0,1] \to [0,1]$ is a maximal cDFF if and only if the following conditions hold:
\begin{enumerate}[(i)]
  \item[(i)] 
  $\phi$ is superadditive.
        \item[(ii)] 
    $\phi$ is symmetric in the sense $\phi(x)+\phi(1-x)=1$.
  \item[(iii)] 
    $\phi(0)=0$.
  \end{enumerate}
\end{theorem}


\begin{theorem}[{\cite[Theorem 3.1]{alves-clautiaux-valerio-rietz-2016:dual-feasible-book}}]
\label{thm:maximality-general}
Let $\phi\colon \mathbb{R} \to\mathbb{R}$ be a given function. If $\phi$ satisfies the following conditions, then $\phi$ is a maximal gDFF:
\begin{enumerate}[(i)]
  \item[(i)] 
  $\phi$ is superadditive.
        \item[(ii)] 
    $\phi$ is symmetric in the sense $\phi(x)+\phi(1-x)=1$.
  \item[(iii)] 
    $\phi(0)=0$.
    \item[(iv)]
    There exists an $\epsilon>0$ such that $\phi(x)\ge 0$ for all $x\in(0, \epsilon)$.
  \end{enumerate}
If $\phi$ is a maximal gDFF, then $\phi$ satisfies conditions $(i)$, $(iii)$ and $(iv)$.
\end{theorem}

\begin{remark}
The function $\phi(x)=cx$ for $0\le c < 1$ is a maximal gDFF but it does not satisfy 
condition $(ii)$.
\end{remark}

The following two propositions indicate additional properties of maximal gDFFs. Proposition \ref{limit-slope} shows that any maximal gDFF is the sum of a linear function and a bounded function. Proposition \ref{given-point} explains the behavior of nonlinear maximal gDFFs at given points.

\begin{proposition}[{\cite[Proposition 3.4]{alves-clautiaux-valerio-rietz-2016:dual-feasible-book}}]
\label{limit-slope}
If $\phi\colon \mathbb{R}\to\mathbb{R}$ is a maximal gDFF and $t=\sup\{\frac{\phi(x)}{x}: x>0\}.$ Then we have $\lim_{x\to \infty} \frac{\phi(x)}{x}=t \le -\phi(-1),$ and for any $x\in\mathbb{R}$, it holds that: $tx-\max\{0,t-1\}\le\phi(x)\le tx.$
\end{proposition}

\begin{proposition}[{\cite[Proposition 3.5]{alves-clautiaux-valerio-rietz-2016:dual-feasible-book}}]
\label{given-point}
If $\phi\colon \mathbb{R}\to\mathbb{R}$ is a maximal gDFF and not
of the kind  $\phi(x)=cx$ for $0\le c < 1$, then $\phi(1)=1$ and $\phi(\frac{1}{2})=\frac{1}{2}$.
\end{proposition}

Maximal gDFFs can be obtained by extending maximal cDFFs to the domain $\mathbb{R}$.  \cite[Proposition 3.10]{alves-clautiaux-valerio-rietz-2016:dual-feasible-book} uses quasiperiodic extensions and  \cite[Proposition 3.12]{alves-clautiaux-valerio-rietz-2016:dual-feasible-book}  uses affine functions when $x$ is not in $[0,1]$.

The following proposition utilizes the fact that maximal gDFFs are superadditive and nondecreasing, which can be used to generate valid inequalities for general linear integer optimization problems.

\begin{proposition}[{\cite[Proposition 5.1] {alves-clautiaux-valerio-rietz-2016:dual-feasible-book}}]
If $\phi$ is a maximal gDFF and $L=\{x\in\mathbb{Z}_+^n: \sum_{j=1}^{n}a_{ij}x_j\le b_j, i=1,2,\dots,m\}$, then for any $i$, $\sum_{j=1}^{n}\phi(a_{ij})x_j\le \phi(b_j)$ is a valid inequality for $L$.
\end{proposition}

In the book \cite{alves-clautiaux-valerio-rietz-2016:dual-feasible-book}, the authors include several known gDFFs. We will use the term ``piecewise linear" throughout the paper without explanation. We refer readers to \cite{hong-koeppe-zhou:software-paper} for precise definitions of  ``piecewise linear" functions in both continuous and discontinuous cases. Although piecewise linearity is not implied in the definition of gDFF, nearly all known gDFFs are piecewise linear. 
\smallbreak

\section{Characterization of maximal general DFFs}
\label{s:characterization}

Alves et al. \cite{alves-clautiaux-valerio-rietz-2016:dual-feasible-book} provided several sufficient conditions and necessary conditions of maximal gDFFs in Theorem \ref{thm:maximality-general}, but they do not match precisely. Inspired by the characterization of minimal 
cut-generating functions in the Y{\i}ld{\i}z--Cornu\'ejols model \cite{yildiz2016cut}, we complete the characterization
of maximal gDFFs. 
\begin{proposition}
A function $\phi\colon \mathbb{R}\to\mathbb{R}$ is a maximal gDFF if and only if the following conditions hold:
 \begin{enumerate}[(i)]
  \item[(i)] 
  $\phi(0)=0$.
      \item[(ii)] 
   $\phi$ is superadditive.
   \item[(iii)] 
   $\phi(x)\ge 0$ for all $x\in \mathbb{R}_+$.
   \item[(iv)]
   $\phi$ satisfies the generalized symmetry condition in the sense  \\$\phi(r)=\inf_{k}\{\frac{1}{k}(1-\phi(1-kr)): k\in \mathbb{Z}_+\}$.
    \end{enumerate}

\end{proposition}

\begin{proof}
Suppose $\phi$ is a maximal gDFF, then conditions $(i),(ii),(iii)$ hold by \autoref{thm:maximality-general}. For any $r\in\mathbb{R}$ and $k\in \mathbb{Z}_+$, $kr+(1-kr)=1\Rightarrow k\phi(r)+\phi(1-kr)\le 1$. So $\phi(r)\le \frac{1}{k}(1-\phi(1-kr))$ for any positive integer $k$, then $\phi(r)\le\inf_{k}\{\frac{1}{k}(1-\phi(1-kr)): k\in \mathbb{Z}_+\}$. 

If there exists $r_0$ such that $\phi(r_0)<\inf_{k}\{\frac{1}{k}(1-\phi(1-kr_0))\colon k\in \mathbb{Z}_+\}$, then define a function $\phi_1$ which takes value $\inf_{k}\{\frac{1}{k}(1-\phi(1-kr_0))\colon k\in \mathbb{Z}_+\}$ at $r_0$ and $\phi(r)$ if $r\neq r_0$. We claim that $\phi_1$ is a gDFF which dominates $\phi$. 
Given a function $y\colon\mathbb{R} \to \mathbb{Z}_+, \,\text{and $y$ has finite support}$ satisfying $\sum_{r\in\mathbb{R}}r\,y(r)\le 1$.  
$\sum_{r\in\mathbb{R}}\phi_1(r)\,y(r)=\phi_1(r_0)\,y(r_0)+\sum_{r\neq r_0} \phi(r)\,y(r)$. If $y(r_0)=0$, then it is clear that $\sum_{r\in\mathbb{R}}\phi_1(r)\,y(r)\le 1$. Let $y(r_0)\in\mathbb{Z}_+$, then $\phi_1(r_0)\le \frac{1}{y(r_0)}(1-\phi(1-y(r_0)\,r_0))$ by definition of $\phi_1$, then 
\begin{equation*}
\phi_1(r_0)\, y(r_0)+\phi(1-y(r_0)\,r_0)\le 1
\end{equation*}
From the superadditive condition and increasing property, we get 
\begin{equation*}
\sum_{r\neq r_0} \phi(r)\,y(r)\le \phi(\sum_{r\neq r_0} r\,y(r)) \le \phi(1-y(r_0)\,r_0)
\end{equation*}
From the two inequalities we can conclude that $\phi_1$ is a gDFF and dominates $\phi$, which contradicts the maximality of $\phi$. Therefore, the condition $(iv)$ holds.

Suppose there is a function $\phi\colon \mathbb{R}\to\mathbb{R}$ satisfying all four conditions. Choose $r=1$ and $k=1$, we can get $\phi(1)\le 1$. Together with conditions $(i),(ii),(iii)$, it guarantees that $\phi$ is a gDFF. Assume that there is a gDFF $\phi_1$ dominating $\phi$ and there exists $r_0$ such that $\phi_1(r_0)>\phi(r_0)=\inf_{k}\{\frac{1}{k}(1-\phi(1-kr_0))\colon k\in \mathbb{Z}_+\}$. So there exists some $k\in\mathbb{Z}_+$ such that 
\begin{align*}
& \, \, \phi_1(r_0)>\frac{1}{k}(1-\phi(1-kr_0)) \\
\Leftrightarrow & \, \, k\phi_1(r_0)+\phi(1-kr_0)>1 \\
\Rightarrow & \, \, k\phi_1(r_0)+\phi_1(1-kr_0)>1
\end{align*}
The last step contradicts the fact that $\phi_1$ is a gDFF, since $kr_0+(1-kr_0)=1$. Therefore, $\phi$ is a maximal gDFF.
\end{proof}

Parallel to the restricted minimal and strongly minimal functions in the Y{\i}ld{\i}z--Cornu\'ejols model \cite{yildiz2016cut}, ``restricted maximal" and ``strongly maximal" gDFFs are defined by strengthening the notion of maximality. 

\begin{definition}
We say that a gDFF $\phi$ is \emph{implied via scaling} by a gDFF $\phi_1$, if $\beta\phi_1\ge \phi$ for some $0\le\beta\le 1$.
We call a gDFF $\phi\colon \mathbb{R}\to \mathbb{R}$ \emph{restricted maximal} if $\phi$ is not implied via scaling by a distinct gDFF $\phi_1$.
\end{definition}

\begin{definition}
We say that a gDFF $\phi$ is \emph{implied} by a gDFF $\phi_1$, if $\phi(x)\le \beta\phi_1(x)+\alpha x$ for some $0\le \alpha,\beta \le1$ and $\alpha+\beta\le 1$.
We call a gDFF $\phi\colon \mathbb{R}\to \mathbb{R}$ \emph{strongly maximal} if $\phi$ is not implied by a distinct gDFF $\phi_1$.
\end{definition}

Note that restricted maximal gDFFs are maximal and strongly maximal gDFFs are restricted maximal. We can simply choose $\beta=1$ and $\beta=1, \alpha=0 $, respectively. Based on the definition of strong maximality, $\phi(x)=x$ is implied by the zero function, so $\phi$ is not strongly maximal, though it is extreme. The characterizations of restricted maximal and strongly maximal gDFFs only involve the standard symmetry condition instead of the generalized symmetry condition.

\begin{theorem}
A function $\phi\colon \mathbb{R}\to\mathbb{R}$ is a restricted maximal gDFF if and only if the following conditions hold:
 \begin{enumerate}[(i)]
  \item[(i)] 
  $\phi(0)=0$.
      \item[(ii)] 
   $\phi$ is superadditive.
   \item[(iii)] 
   $\phi(x)\ge 0$ for all $x\in \mathbb{R}_+$.
   \item[(iv)]
   $\phi(x)+\phi(1-x)=1$.
    \end{enumerate}

\end{theorem}

\begin{proof}
It is easy to show that $\phi$ is valid and restricted maximal if $\phi$ satisfies conditions $(i-iv)$.
Suppose $\phi$ is a restricted maximal gDFF, then we only need to prove condition $(iv)$, since restricted maximality implies maximality.

Suppose there exists some $x$ such that $\phi(x)+\phi(1-x)<1$. By the characterization of maximality, $\phi(x)=\inf_{k}\{\frac{1}{k}(1-\phi(1-kx))\colon k\in \mathbb{Z}_+\}$.

\emph{Case 1:} there exists some $k\in \mathbb{N}$ such that $\phi(x)=\frac{1}{k}(1-\phi(1-kx))$. By superadditivity $k\phi(x)=1-\phi(1-kx)=1-\phi(1-x-(k-1)x)\ge 1-\phi(1-x)+\phi((k-1)x)\ge 1-\phi(1-x)+(k-1)\,\phi(x)$, which implies $\phi(x)+\phi(1-x)\ge1$, in contradiction to the assumption above.

\emph{Case 2:} for any $\epsilon>0$, there exists a corresponding $k_\epsilon\in \mathbb{N}$, such that 
$$\phi(x)<\frac{1}{k_\epsilon}(1-\phi(1-k_\epsilon x))<\phi(x)+\epsilon$$
Then $\phi(k_\epsilon x)\le 1-\phi(1-k_\epsilon x)<k_\epsilon\phi(x)+k_\epsilon\epsilon$ or equivalently $\frac{\phi(k_\epsilon x)}{k_\epsilon}<\phi(x)+\epsilon$. Since $\phi$ is superadditive, $\phi(x)\le\frac{\phi(k_\epsilon x)}{k_\epsilon}$. Let $\epsilon$ go to $0$ in the inequality $\phi(x)\le\frac{\phi(k_\epsilon x)}{k_\epsilon}<\phi(x)+\epsilon$, and we have $\lim_{\epsilon\to 0}\frac{\phi(k_\epsilon x)}{k_\epsilon}=\phi(x)$.
It is easy to see that $\lim_{\epsilon\to 0}k_\epsilon=+\infty$.

Next, we will show that $\phi(kx)=k\phi(x)$ for any positive integer $k$. Suppose $\bar{k}$ is the smallest integer such that $\frac{\phi(\bar{k}x)}{\bar{k}}=\phi(x)+\delta$ for some $\delta>0$. Then for any $i\ge \bar{k}$, there exist $\lambda_i, r_i\in\mathbb{Z_+}$, such that $i=\lambda_i\bar{k}+r_i$, $0\le r_i<\bar{k}$. Then 
\begin{align*}
 \phi(ix) = & \, \, \phi(\lambda_i\bar{k}x+r_ix)
\ge  \lambda_i\phi(\bar{k}x)+\phi(r_ix)\\
\ge & \, \, \lambda_i\bar{k}\phi(x)+\lambda_i\bar{k}\delta+r_i\phi(x)
=  i\phi(x)+ (i-r_i)\delta
 \end{align*}
 Therefore $\frac{\phi(ix)}{i}\ge \phi(x)+\delta-\frac{r_i}{i}\delta$ for any $i\ge \bar{k}$. Since $r_i$ is bounded, $\frac{\phi(ix)}{i}\ge \phi(x)+\frac{\delta}{2}$ for any $i\ge 2\bar{k}$, which contradicts $\lim_{\epsilon\to 0}\frac{\phi(k_\epsilon x)}{k_\epsilon}=\phi(x)$. Therefore $\phi(kx)=k\phi(x)$ for any positive integer $k$.
 From Proposition \ref{given-point} we know $\phi(1)=1$.
\begin{align*}
 k\phi(x)=& \, \, \phi(kx)\ge (k-1)\phi(1)+\phi(1-k(1-x))\\
\Leftrightarrow & \, \, 1-\phi(x) \le  \frac{1-\phi(1-k(1-x))}{k}\\
\Rightarrow & \, \, 1-\phi(x)\le  \inf_k\frac{1-\phi(1-k(1-x))}{k}=\phi(1-x)
\end{align*} 
The above inequality contradicts our original assumption.

In both cases, we have a contradiction if $\phi(x)+\phi(1-x)<1$. Therefore, $\phi(x)+\phi(1-x)=1$, which completes the proof.
\end{proof}

\begin{remark}
\label{rmk}
Let $\phi$ be a maximal gDFF that is not linear,
we know that $\phi(1)=1$ from Proposition \ref{given-point}. 
If $\phi$ is implied via scaling by a gDFF $\phi_1$, or equivalently $\beta\phi_1\ge \phi$ for some $0\le\beta\le 1$,
 then $\beta\phi_1(1)\ge \phi(1)$. Then $\beta=1$ and $\phi$ is dominated by $\phi_1$. The maximality of $\phi$ implies $\phi=\phi_1$, so $\phi$ is restricted maximal. Therefore, we have a simpler version of characterization of maximal gDFFs. 
\end{remark}

\begin{theorem}
A function $\phi\colon \mathbb{R}\to \mathbb{R}$ is a maximal gDFF if and only if the following conditions hold:
 \begin{enumerate}[(i)]
  \item[(i)] 
  $\phi(0)=0$.
      \item[(ii)] 
   $\phi$ is superadditive.
   \item[(iii)] 
   $\phi(x)\ge 0$ for all $x\in \mathbb{R}_+$.
   \item[(iv)]
   $\phi(x)+\phi(1-x)=1$ or $\phi(x)=a x$, $0\le a< 1$.
    \end{enumerate}

\end{theorem}

\begin{theorem}
\label{strongly-maximal}
A function $\phi\colon \mathbb{R}\to \mathbb{R}$ is a strongly maximal gDFF if and only if $\phi$ is a restricted maximal gDFF and $\lim_{\epsilon\to0^+}\frac{\phi(\epsilon)}{\epsilon}=0$.
\end{theorem}

\begin{proof}
To prove the ``only if" part, we only need to show $\lim_{\epsilon\to0^+}\frac{\phi(\epsilon)}{\epsilon}=0$ for a strongly maximal gDFF $\phi$. We first show that $\liminf_{\epsilon\to0^+}\frac{\phi(\epsilon)}{\epsilon}=0$. It is clear  that $\liminf_{\epsilon\to0^+}\frac{\phi(\epsilon)}{\epsilon}\ge0$ since $\phi$ is restricted maximal. Assume  $\liminf_{\epsilon\to0^+}\frac{\phi(\epsilon)}{\epsilon}=s>0$, then there exist $\delta>0$ and $s'<s$ (small enough) such that $\phi(x)\ge  s'x$ for $x\in [0,\delta]$.  Define a new function $\phi_1(x)=\frac{\phi(x)-s'x}{1-s'}$ and $\phi$ is implied by $\phi_1$. Note that $\phi_1$ is a restricted maximal gDFF. The strong maximality of $\phi$ implies $\phi_1(x)=\phi(x)=x$. Therefore, $\phi(x)=x$ is not strongly maximal.

Next we show that $\lim_{\epsilon\to0^+}\frac{\phi(\epsilon)}{\epsilon}=0$. Suppose $\limsup_{\epsilon\to0^+}\frac{\phi(\epsilon)}{\epsilon}=3s>0$. There exist two positive and decreasing sequences $(x_n)_{n=1}^{\infty}$ and $(y_n)_{n=1}^{\infty}$ approaching $0$, such that $\phi(x_n)> 2sx_n$ and $\phi(y_n)< sy_n$. Fix $y_1$ and choose $0<x_n<y_1$ and $k\in\mathbb{Z}_{++}$ such that $y_1\ge kx_n\ge \frac{y_1}{2}$. Since $\phi$ is  superadditive and nondecreasing, $\phi(y_1)\ge\phi(kx_n)\ge k\phi(x_n)>2ksx_n\ge sy_1$, which contradicts the choice of $y_1$. $\limsup_{\epsilon\to0^+}\frac{\phi(\epsilon)}{\epsilon}=\liminf_{\epsilon\to0^+}\frac{\phi(\epsilon)}{\epsilon}=0$, so $\lim_{\epsilon\to0^+}\frac{\phi(\epsilon)}{\epsilon}=0$ for a strongly maximal gDFF $\phi$.

As for the ``if" part, we assume $\phi$ is restricted maximal and $\lim_{\epsilon\to0^+}\frac{\phi(\epsilon)}{\epsilon}=0$. Suppose $\phi$ is implied by a gDFF $\phi_1$ meaning $\phi(x)\le \beta\phi_1(x)+\alpha x$ and $\beta,\alpha\ge 0, \beta+\alpha\le 1$. Let $x=1$, $1\le \beta \phi_1(1)+\alpha\le \beta+\alpha\le 1$. We know that $\beta=1-\alpha$. Note that $\beta\phi_1(x)+\alpha x$ is also a gDFF (a convex combination of two gDFFs $\phi_1$ and $x$), then $\phi(x)= (1-\alpha)\phi_1(x)+\alpha x$ due to maximality of $\phi$. Divide by $x$ from the above equation and take the $\liminf$ as $x\to 0^+$, we can conclude $\alpha=0$. So $\phi$ is strongly maximal.

\end{proof}

The following theorem indicates that maximal, restricted maximal and strongly maximal gDFFs  exist, and they are potentially stronger than just valid gDFFs. The proof is analogous to the proof of \cite[Theorem 1, Proposition 6, Theorem 9]{yildiz2016cut} and is therefore omitted.
 (See Appendix \ref{s:proof-existence} for the proof.)
\begin{theorem}
\label{existence}
\begin{enumerate}[(i)]
  \item[(i)] 
 Every gDFF is dominated by a maximal gDFF.
       \item[(ii)] 
  Every gDFF is implied via scaling by a restricted maximal gDFF.
  \item[(iii)]
     Every nonlinear gDFF is implied by a strongly maximal gDFF.
       \end{enumerate}
\end{theorem}

\smallbreak

\section{Relation to cut-generating functions}
\label{s:YC}

We define an infinite dimensional space $Y$ called ``the space of nonbasic variables" as $Y=\{y : \,y\colon\mathbb{R} \to \mathbb{Z}_+ \, \text{and $y$ has finite support}\}$, and we refer to the zero function as the origin of $Y$. In this section, we study valid inequalities of certain subsets of the space $Y$ and connect gDFFs to a particular family of cut-generating functions.

In the paper of Y\i{}ld\i{}z and Cornu\'ejols \cite{yildiz2016cut}, the authors considered the following generalization of the Gomory--Johnson model:

\begin{equation}
\label{a:1}
x=f+\sum_{r\in\mathbb{R}}r\, y(r)
\end{equation}
$$x\in S,\, y:\mathbb{R} \to \mathbb{Z}_+, \,\text{and $y$ has finite support.}$$
where $S$ can be any nonempty subset of $\mathbb{R}$. 
A function $\pi\colon \mathbb{R}\to \mathbb{R}$ is called a \emph{valid cut-generating function} if the inequality $\sum_{r\in\mathbb{R}}\pi(r)\, y(r)\ge 1$ holds for all feasible solutions $(x,y)$ to (\ref{a:1}). In order to ensure that such cut-generating functions exist, they only consider the case $f\notin S$. Otherwise, if $f\in S$, then $(x,y)=(f,0)$ is a feasible solution and there is no function $\pi$ which can make the inequality $\sum_{r\in\mathbb{R}}\pi(r)\, y(r)\ge 1$ valid. Note that $y\in Y$ for any feasible solution $(x,y)$ to (\ref{a:1}), and  all valid inequalities in the form of $\sum_{r\in\mathbb{R}}\pi(r)\, y(r)\ge 1$ to (\ref{a:1}) are inequalities which separate the origin of $Y$.

We consider two different but related models in the form of (\ref{a:1}). Let $f=-1$, $S=\{0\}$, and the feasible region $Y_{=1}=\{y : \sum_{r\in\mathbb{R}}r\, y(r)=1, \, \, y\colon\mathbb{R} \to \mathbb{Z}_+ \, \text{and $y$ has finite support}\}$. Let $f=-1$, $S=(-\infty,0]$, and the feasible region $Y_{\le1}=\{y : \sum_{r\in\mathbb{R}}r\, y(r)\le1, \, \, y\colon\mathbb{R} \to \mathbb{Z}_+ \, \text{and $y$ has finite support}\}$. It is immediate to check that the latter model is the relaxation of the former. Therefore $Y_{=1}\subsetneq Y_{\le1}$ and any valid inequality for $Y_{\le1}$ is also valid for $Y_{=1}$.

Jeroslow \cite{jeroslow1979minimal}, Blair \cite{BLAIR1978147} and Bachem et al. \cite{BACHEM198263} studied minimal valid inequalities of the set $Y_{=b}=\{y : \sum_{r\in\mathbb{R}}r\, y(r)=b, \, \, y\colon\mathbb{R} \to \mathbb{Z}_+ \, \text{and $y$ has finite support}\}$. Note that $Y_{=b}$ is the set of feasible solutions to (\ref{a:1}) for $S=\{0\}$, $f=-b$.
The notion ``minimality" they used is in fact the restricted minimality in the Y{\i}ld{\i}z--Cornu\'ejols model.  In this section, we use the terminology introduced by Y\i{}ld\i{}z and Cornu\'ejols.
Jeroslow \cite{jeroslow1979minimal} showed that finite-valued subadditive (restricted minimal) functions are sufficient to generate all necessary valid inequalities of $Y_b$ for bounded mixed integer programs. K{\i}l{\i}n\c{c}-Karzan and Yang \cite{KKYang15} discussed whether finite-valued functions are sufficient to generate all necessary inequalities for the convex hull description of disjunctive sets. Interested readers are referred to \cite{KKYang15} for more details on the sufficiency question. Blair \cite{BLAIR1978147} extended Jeroslow's result to rational mixed integer programs. Bachem et al. \cite{BACHEM198263} characterized restricted minimal cut-generating functions under some continuity assumptions, and they showed that restricted minimal
functions satisfy the symmetry condition.


In terms of the relaxation $Y_{\le 1}$, gDFFs can generate the valid inequalities in the form of $\sum_{r\in\mathbb{R}}\phi(r)\, y(r)\le 1$, and such inequalities do not separate the origin. Note that there is no valid inequality separating the origin since $0\in Y_{\le1}$. 



Cut-generating functions provide valid inequalities which separate the origin for $Y_{=1}$, but such inequalities are not valid for $Y_{\le1}$.  In terms of inequalities do not separate the origin, any inequality in the form of $\sum_{r\in\mathbb{R}}\phi(r)\, y(r)\le 1$ generated by some gDFF $\phi$ is valid for $Y_{\le1}$ and hence valid for $Y_{=1}$, since the model of $Y_{\le1}$ is the relaxation of that of $Y_{=1}$. There also exist valid inequalities which do not separate the origin for $Y_{=1}$ but are not valid for $Y_{\le1}$. For instance, $\sum_{r\in\mathbb{R}}-r\, y(r)\le 1$ is valid for $Y_{=1}$ but not valid for $Y_{\le1}$. For any $y\in Y_{=1}$, $\sum_{r\in\mathbb{R}}-r\, y(r)=-1\le 1$. Consider a feasible solution $y\in Y_{\le1}$ where $y(-1)=2$, $y(r)=0$ if $r\neq -1$, then $\sum_{r\in\mathbb{R}}-r\, y(r)=2>1$.

Y\i{}ld\i{}z and Cornu\'ejols \cite{yildiz2016cut} introduced the notions of minimal, restricted minimal and strongly minimal cut-generating functions. We call the cut-generating functions to the model (\ref{a:1}) when $f=-1$, $S=\{0\}$  cut-generating functions for $Y_{=1}$, and we restate the definitions of minimality of such cut-generating functions. A valid cut-generating function $\pi$ is called \emph{minimal} if it does not dominate another valid cut-generating function $\pi'$. A cut-generating
function $\pi'$ implies a cut-generating function $\pi$ via scaling if there exists $\beta\ge 1$ such
that $\pi\ge \beta \pi'$.  A valid cut-generating function $\pi$ is \emph{restricted minimal} if there is no another cut-generating function $\pi'$ implying $\pi$ via scaling. A cut-generating
function $\pi'$ implies a cut-generating function $\pi$ if there exist $\alpha, \beta$, and $\beta\ge 0, \alpha+\beta\ge 1$ such that $\pi(x)\ge \beta \pi'(x)+\alpha x$. A valid cut-generating function $\pi$ is \emph{strongly minimal} if there is no another cut-generating function $\pi'$ implying $\pi$. Y\i{}ld\i{}z and Cornu\'ejols also characterized minimal and restricted minimal functions without additional assumptions.
As for the strong minimality and extremality, they mainly focused on the case where $f\in \overline{\conv(S)}$ and $ \overline{\conv(S)}$ is full-dimensional. We discuss the strong minimality and extremality when $f=-1$, $S=\{0\}$ in \autoref{distinction}.

In the rest of this section, we show that gDFFs are closely related to cut-generating functions for $Y_{=1}$. The main idea is that valid inequalities generated by cut-generating functions for $Y_{=1}$ can be lifted to valid inequalities generated by gDFFs for the relaxation $Y_{\le1}$.

We include the characterizations \cite[Theorem 2, Proposition 5]{yildiz2016cut} of minimal and restricted minimal cut-generating functions for $Y_{=1}$ below. Bachem et al. had the same characterization \cite[Theorem]{BACHEM198263} as Theorem \ref{YC:2} under continuity assumptions at the origin. 

\begin{theorem}
\label{YC:1}
A function $\pi\colon \mathbb{R}\to \mathbb{R}$ is a minimal cut-generating function for $Y_{=1}$ if and only if $\pi(0)=0$, $\pi$ is subadditive, and $\pi(r)=\sup_{k}\{\frac{1}{k}(1-\pi(1-kr)): k\in \mathbb{Z}_+\}$.
\end{theorem}

\begin{theorem}
\label{YC:2}
A function $\pi\colon \mathbb{R}\to \mathbb{R}$ is a restricted minimal cut-generating function for $Y_{=1}$ if and only if $\pi$ is minimal and $\pi(1)=1$.
\end{theorem}

The following theorem describes the conversion between gDFFs and cut-generating functions for $Y_{=1}$. We omit the proof which is a straightforward computation, utilizing the characterization of (restricted) maximal gDFFs and (restricted) minimal cut-generating functions. (See Appendix \ref{s:proof-conversion} for the proof.)

\begin{theorem}
\label{valid-conversion}
Given a valid/maximal/restricted maximal gDFF $\phi$, then for every $0< \lambda < 1$, the following function is a valid/minimal/restricted minimal cut-generating function for $Y_{=1}$:
$$\pi_\lambda(x)=\frac{x-(1-\lambda)\,\phi(x)}{\lambda}$$
Given a valid/minimal/restricted minimal cut-generating function $\pi$ for $Y_{=1}$, which is Lipschitz continuous at $x=0$, then there exists $\delta>0$ such that for all $0<\lambda<\delta$ the following function is a valid/maximal/restricted maximal gDFF:
$$\phi_\lambda(x)=\frac{x-\lambda\,\pi(x)}{1-\lambda}, \quad 0< \lambda < 1$$
\end{theorem}

\smallbreak

\begin{remark}
\label{distinction}
We discuss the distinctions between these two family of functions.
\begin{enumerate}[(i)]
  \item[(i)] 
It is not hard to prove that extreme gDFFs are always maximal. However, unlike cut-generating functions for $Y_{=1}$, extreme gDFFs are not always restricted maximal. $\phi(x)=0$ is an extreme gDFF but not restricted maximal. 
       \item[(ii)] 
  By applying the proof of \cite[Proposition 28]{yildiz2016cut}, we can show that no strongly minimal cut-generating function for $Y_{=1}$ exists. However, there exist strongly maximal gDFFs by Theorem \ref{existence}. Moreover, we can use the same conversion formula in Theorem \ref{valid-conversion} to convert a restricted minimal cut-generating function to a strongly maximal gDFF (see Theorem \ref{strongly-conversion} below). In fact, it suffices to choose a proper $\lambda$ such that $\lim_{\epsilon\to0^+}\frac{\phi_\lambda(\epsilon)}{\epsilon}=0$ by the characterization of strongly maximal gDFFs (Theorem \ref{strongly-maximal}).
  \item[(iii)]
     There is no extreme piecewise linear cut-generating function $\pi$ for $Y_{=1}$ which is Lipschitz continuous at $x=0$, except for $\pi(x)=x$. If $\pi$ is such an extreme function, then for any $\lambda$ small enough, we claim that $\phi_\lambda$ is an extreme gDFF. Suppose $\phi_\lambda=\frac{1}{2}\phi^1+\frac{1}{2}\phi^2$ and let $\pi_\lambda^1,\pi_\lambda^2$ be the corresponding cut-generating functions of $\phi^1,\phi^2$ by Theorem \ref{valid-conversion}. Note that $\pi=\frac{1}{2}(\pi_\lambda^1+\pi_\lambda^2)$, which implies $\pi=\pi_\lambda^1=\pi_\lambda^2$ and $\phi_\lambda=\phi_\lambda^1=\phi_\lambda^2$.  Thus $\phi_\lambda$ is extreme. By Lemma \ref{lemma2} and the extremality of $\phi_\lambda$, we know $\phi_\lambda(x)=x$ or there exists $\epsilon>0$, such that $\phi_\lambda(x)=0$ for $x\in [0,\epsilon)$. If $\phi_\lambda(x)=x$, then $\pi(x)=x$. Otherwise, $\lim_{x\to0^+}\frac{\phi_\lambda(x)}{x}=0$ for any small enough $\lambda$. 
$$0=\lim_{x\to0^+}\frac{\phi_\lambda(x)}{x}=\lim_{x\to0^+}\frac{x-\lambda\pi(x)}{(1-\lambda)x}=\frac{1-\lambda\lim_{x\to0^+}\frac{\pi(x)}{x}}{1-\lambda}$$     
  The above equation implies $\lim_{x\to0^+}\frac{\pi(x)}{x}=\frac{1}{\lambda}$ for any small enough $\lambda$, which is not possible.   Therefore, $\pi$ cannot be extreme except for $\pi(x)=x$.

       \end{enumerate}
\end{remark}

\begin{theorem}
\label{strongly-conversion}
Given a non-linear restricted minimal cut-generating function $\pi$ for $Y_{=1}$, which is Lipschitz continuous at $0$, then there exists $\lambda>0$ such that the following function is a strongly maximal gDFF:
$$\phi_\lambda(x)=\frac{x-\lambda\,\pi(x)}{1-\lambda}$$
\end{theorem}

\smallbreak

\section{2-slope theorem}
\label{s:2slope}
In this section, we prove a 2-slope theorem for extreme gDFFs, in the spirit of the 2-slope theorem of Gomory and Johnson \cite{infinite,infinite2}. First we introduce two lemmas showing that extreme gDFFs have certain structures.  
\begin{lemma}
\label{lemma1}
Piecewise linear maximal gDFFs are continuous at $0$ from the right.
\end{lemma}

\begin{proof}
The claim follows directly from superadditivity.
\end{proof}



\begin{lemma}
\label{lemma2}
Let $\phi$ be a piecewise linear extreme gDFF. 
 \begin{enumerate}[(i)]
  \item[(i)] If $\phi$ is strictly increasing, then $\phi(x)=x$.
       \item[(ii)] If $\phi$ is not strictly increasing, then there exists $\epsilon>0$, such that $\phi(x)=0$ for $x\in [0,\epsilon)$.
      \end{enumerate}
\end{lemma}

\begin{proof}
We provide a proof sketch. (See Appendix \ref{s:proof-lemma2} for the complete proof.) 

From Lemma \ref{lemma1} we assume
$\phi(x)=sx$, $x\in[0,x_1)$ and $s>0$. We claim $0\le s<1$ due to maximality of $\phi$ and $\phi(1)=1$. Define a function: $\phi_1(x)=\frac{\phi(x)-sx}{1-s}$, and it is straightforward to show that $\phi_1$ is maximal, and $\phi(x)=sx+(1-s)\phi_1(x)$. From the extremality of $\phi$, $s=0$ or $\phi(x)=x$.

\end{proof}

From Lemma \ref{lemma2}, we know $0$ must be one slope value of a piecewise linear extreme gDFF $\phi$, except for $\phi(x)=x$. Now we prove the 2-slope theorem for extreme gDFFs. The fundamental tool in the proof is the Interval Lemma \cite[Lemma 2.2]{bhk-IPCOext}, which was used in the proof of Gomory--Johnson's 2-slope theorem. We include one version of the Interval Lemma here.

\begin{lemma}[Interval Lemma]
Let $a_1 < a_2$ and $b_1 < b_2$. Consider the intervals $A = [a_1, a_2]$,
$B = [b_1, b_2]$, and $A+B = [a_1+b_1, a_2+b_2]$. Let $f \colon A \to \mathbb{R}$, $g \colon B \to \mathbb{R}$, and $h \colon A+B \to \mathbb{R}$ be bounded functions on $A$, $B$ and $A+B$, respectively. If $f(a) + g(b) = h(a + b)$ for all $a \in A$ and $b \in B$, then $f$, $g$,
and $h$ are affine functions with identical slopes in the intervals $A$, $B$, and $A + B$, respectively.
\end{lemma}

\begin{theorem}
\label{t:2-slope}
Let $\phi$ be a continuous piecewise linear strongly maximal gDFF with only 2 slope values,  then $\phi$ is extreme.
\end{theorem}

\begin{proof}
Since $\phi$ is strongly maximal with 2 slope values, we know one slope value must be $0$. Suppose $\phi=\frac{1}{2}(\phi_1+\phi_2)$, where $\phi_1, \phi_2$ are two maximal gDFFs. From Proposition \ref{given-point}, we know $\phi(1)=1$, which implies $\phi_1(1)=\phi_2(1)=1$. Let $s$ be the other slope value of $\phi$. Due to superadditivity of $\phi$, there exist $\epsilon, \delta>0$ such that $\phi(x)=sx$ for $x\in[-\epsilon,0]$ and $\phi(x)=0$ for $x\in[0,\delta]$. We want to 
to show $\phi_1, \phi_2$ have slope $0$ where $\phi$ has slope $0$, and  $\phi_1, \phi_2$ have slope $s$ where $\phi$ has slope $s$.

\emph{Case 1:} Suppose $[a,b]$ is a closed interval where $\phi$ has slope value $0$. Choose $\delta'=\min(\delta, \frac{b-a}{2})>0$. Let $I=[0,\delta']$, $J=[a,b-\delta']$, $K=[a,b]$, then $I,J,K$ are three non-empty and proper intervals. Clearly $\phi(x)+\phi(y)=\phi(x+y)$ for $x\in I, y\in J$. Since $\phi_1,\phi_2$ are also superadditive, they satisfy the equality where $\phi$ satisfy the equality. In other words, $\phi_i(x)+\phi_i(y)=\phi_i(x+y)$ for $x\in I$, $y\in J$, $i=1,2$. By Interval Lemma, $\phi_1$ is affine over $[a,b]$ and $[0,\delta']$ with the same slope value $l_1$. Similarly, $\phi_2$ is affine over $[a,b]$ and $[0,\delta']$ with the same slope value $l_2$. It is clear that $l_1=l_2=0$ since $\phi_1,\phi_2$ are  increasing and $0=\frac{1}{2}(l_1+l_2)$.
\smallbreak
\emph{Case 2:} Suppose $[c,d]$ is a closed interval where $\phi$ has slope value $s$. Choose $\epsilon'=\min(\epsilon, \frac{d-c}{2})$. Let $I=[-\epsilon',0]$, $J=[c+\epsilon',d]$, $K=[c,d]$, it is clear that $\phi(x)+\phi(y)=\phi(x+y)$ for $x\in I, y\in J$. Similarly we can prove that $\phi_i$ is affine over $[c,d]$ and $[-\epsilon',0]$ with the same slope value $s_i$ ($i=1,2$).

Consider the interval $[0=x_0, x_1, \dots, x_n=1]$, where $\phi$ has slope $0$ over $[x_k,x_{k+1}]$ with $k$ even and slope $s$ over $[x_k,x_{k+1}]$ with $k$ odd. Then $\phi_i$ have slope $0$ over $[x_k,x_{k+1}]$ with $k$ even and slope $s_i$ over $[x_k,x_{k+1}]$ with $k$ odd. Let $L_0$ and $L_s$ be the total length of intervals where $\phi$ has slope $0$ and $s$, respectively. Then $s\cdot L_s+0\cdot L_0=1$. $\phi_i$ may have possible jumps at breakpoints $x_k$, but it can only jump up since $\phi_i$ is  increasing. Suppose $h_i\ge 0$ are the total jumps of $\phi_i$ at discontinuous points. From $\phi_i(1)=1$ we can obtain the following equation:
$$s_i\cdot L_s +0 \cdot L_0 +h_i=1\,\,\,\,(i=1,2)$$
Note that $s=\frac{1}{2}(s_1+s_2)$ and $s\cdot L_s+0\cdot L_0=1$. So $s_1=s_2=s$ and $h_1=h_2=0$ which implies $\phi_1, \phi_2$ are continuous and $\phi_1=\phi_2=\phi$. Therefore, $\phi$ is extreme.
\end{proof}

\begin{remark}
Alves et al. \cite{alves-clautiaux-valerio-rietz-2016:dual-feasible-book} claimed the following functions by Burdett and Johnson with one parameter $C\ge1$ are maximal gDFFs, where $\{a\}$ represents the fractional part of $a$. 
\begin{equation*}
\phi_{BJ,1}(x;C)=\frac{\floor{Cx}+\max(0,\frac{\{Cx\}-\{C\}}{1-\{C\}})}{\floor{C}}
\end{equation*}
Actually we can prove that they are extreme. If $C\in \bb{N}$, then $\phi_{BJ,1}(x)=x$. If $C\notin \bb{N}$, $\phi_{BJ,1}$ is a continuous 2-slope maximal gDFF with one slope value $0$, therefore it is extreme by Theorem \ref{t:2-slope}. \autoref{fig1} shows two examples of $\phi_{BJ,1}$ and they are constructed by the Python function \sage{phi\textunderscore1\textunderscore bj\textunderscore gdff}\footnote{In this paper, a  function  name  shown  in  typewriter  font  is  the  name  of  the function  in our SageMath program \cite{cutgeneratingfunctionology:online}. At the time of writing, the function is available on the feature branch \sage{gdff}.  Later it will be merged into the \sage{master} branch.}. 

\begin{figure}[!htb]
\minipage{0.45\textwidth}
  \includegraphics[width=\linewidth]{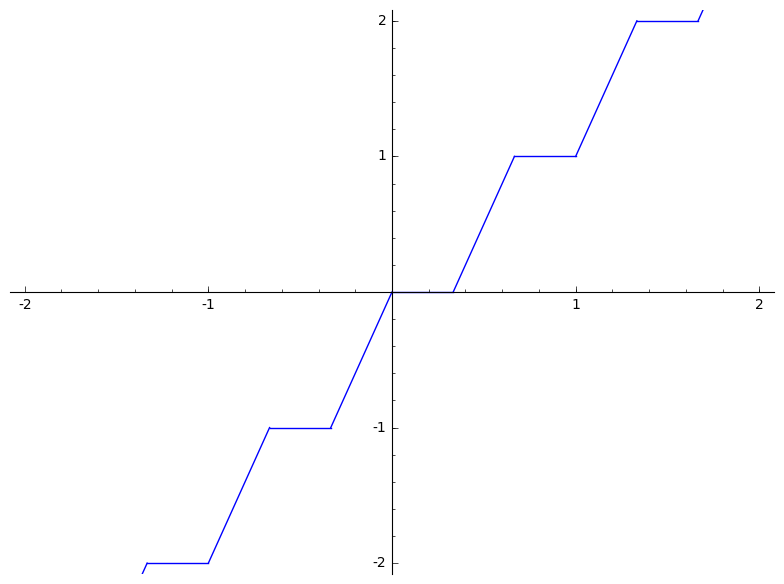}
\endminipage\hfill
\minipage{0.45\textwidth}%
  \includegraphics[width=\linewidth]{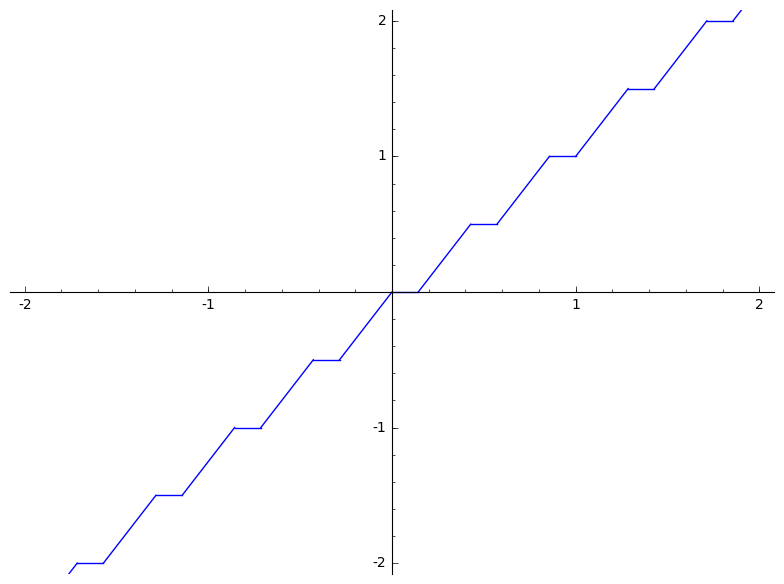}
\endminipage
\caption{GDFFs $\phi_{BJ,1}$ \cite[Example 3.1]{alves-clautiaux-valerio-rietz-2016:dual-feasible-book} for parameter values $C=3/2$ (left) and $C=7/3$ (right). }
\label{fig1}
\end{figure}

\end{remark}

\section{Restricted maximal general DFFs are almost extreme}
\label{s:approximation}
In the previous section, we have shown that any continuous 2-slope strongly maximal gDFF is extreme. In this section, we prove that extreme gDFFs are dense in the set of continuous restricted maximal gDFFs. Equivalently, for any given continuous restricted maximal gDFF $\phi$, there exists an extreme gDFF $\phi_{\mathrm{ext}}$ which approximates $\phi$ as close as desired (with the infinity norm). The idea of the proof is inspired by the approximation theorem of Gomory--Johnson functions \cite{bhm:dense-2-slope}. We first introduce the main theorem in this section. The approximation\footnote{See the constructor \sage{two\textunderscore slope\textunderscore approximation\textunderscore gdff\textunderscore linear}.} is implemented for piecewise linear functions with finitely many pieces.

\begin{theorem}
Let $\phi$ be a continuous restricted maximal gDFF, then for any $\epsilon>0$, there exists an extreme gDFF $\phi_{\mathrm{ext}}$ such that $\|\phi-\phi_{\mathrm{ext}}\|_{\infty}<\epsilon$.
\end{theorem}

\begin{remark}
The result cannot be extended to maximal gDFF. $\phi(x)=ax$ is maximal but not extreme for $0< a<1$.  Any non-trivial extreme gDFF $\phi'$ satisfies $\phi'(1)=1$. $\phi'(1)-\phi(1)=1-a>0$ and $1-a$ is a fixed positive constant. Therefore, $\phi(x)=ax$ cannot be arbitrarily approximated by an extreme gDFF. 
\end{remark}

We briefly explain the structure of the proof. Similar to \cite{basu-hildebrand-koeppe:equivariant,igp_survey,igp_survey_part_2,bhkm}, we
introduce a function 
$\nabla\phi \colon \R \times \R \to \R$, $\nabla\phi(x,y) =
  \phi(x+y) - \phi(x) - \phi(y)$, which measures the
slack in the superadditivity condition. First we approximate a continuous restricted maximal gDFF $\phi$ by a piecewise linear maximal gDFF $\phi_{\mathrm{pwl}}$. Next, we perturb $\phi_{\mathrm{pwl}}$ such that the new maximal gDFF $\phi_{\mathrm{loose}}$ satisfies  $\nabla\phi_{\mathrm{loose}}(x,y)>\gamma>0$ for ``most" $(x,y)\in\mathbb{R}^2$. After applying the 2-slope fill-in procedure to $\phi_{\mathrm{loose}}$, we get a superadditive 2-slope function $\phi_{\mathrm{fill{\text -}in}}$, which is not symmetric anymore. Finally, we symmetrize $\phi_{\mathrm{fill{\text -}in}}$ to get the desired $\phi_{\mathrm{ext}}$.

\begin{lemma}
\label{uni-con}
Any continuous restricted maximal gDFF $\phi$ is uniformly continuous.
\end{lemma}

\begin{proof}
Since $\phi$ is continuous at $0$ and nondecreasing, for any $\epsilon>0$,  there exists $\delta>0$ such that $-\delta<t\le 0$ implies $-\epsilon<\phi(t)\le 0$. For any $x,y$ with $-\delta<x-y<0$, we have $0\ge \phi(x)-\phi(y)\ge\phi(x-y)>-\epsilon$. So $\phi$ is uniformly continuous.

\end{proof}

\begin{lemma}
Let $\phi$ be a continuous restricted maximal gDFF, then for any $\epsilon>0$, there exists a piecewise linear continuous restricted maximal gDFF, such that $\|\phi-\phi_{\mathrm{pwl}}\|_{\infty}<\frac{\epsilon}{3}$.
\end{lemma}

\begin{proof}
By Lemma \ref{uni-con}, $\phi$ is uniformly continuous. For any $\epsilon>0$, there exists $\delta>0$ such that $|x-y|<\delta$ implies $|\phi(x)-\phi(y)|<\frac{\epsilon}{3}$.
Choose $q\in \mathbb{N}$ large enough such that $\frac{1}{q}<\delta$, then $0\le \phi(\frac{n+1}{q})-\phi(\frac{n}{q})<\frac{\epsilon}{3}$ for any integer $n$. We claim that the interpolation of $\phi |_{\frac{1}{q}\mathbb{Z}}$ is the desired $\phi_{\mathrm{pwl}}$. 

We first prove $\|\phi-\phi_{\mathrm{pwl}}\|_{\infty}<\frac{\epsilon}{3}$. For any $x\in\mathbb{R}$, suppose $\frac{n}{q}\le x <\frac{n+1}{q}$ for some integer $n$.  Due to the choice of $q$ and $\delta$, 
$$\phi(x)-\phi_{\mathrm{pwl}}(x)\le \phi(\frac{n+1}{q})-\phi_{\mathrm{pwl}}(\frac{n}{q})=\phi(\frac{n+1}{q})-\phi(\frac{n}{q})<\frac{\epsilon}{3}$$ 
Similarly we can prove $\phi(x)-\phi_{\mathrm{pwl}}(x)>-\frac{\epsilon}{3}$. So $\|\phi-\phi_{\mathrm{pwl}}\|_{\infty}<\frac{\epsilon}{3}$.

Since $\phi |_{\frac{1}{q}\mathbb{Z}}$ is superadditive and satisfies the symmetry condition, then $\phi_{\mathrm{pwl}}$ is also superadditive and satisfies the symmetry condition due to piecewise linearity of $\phi_{\mathrm{pwl}}$. Therefore, $\phi_{\mathrm{pwl}}$ is the desired function. 
\end{proof}

Next, we introduce a parametric family of restricted maximal gDFFs $\phi_{s,\delta}$ which will be used to perturb $\phi_{\mathrm{pwl}}$.
Define

\noindent
\minipage{0.6\textwidth}
\centering
  \begin{equation*}
            \phi_{s,\delta}(x) =
                \begin{cases}
                    sx-s\delta  & \, \,   \text{if $x<-\delta$ } \\
                    2sx & \, \, \text {if $-\delta\le x<0$}\\
                    0 & \, \, \text {if $0\le x<\delta$}\\
                    \frac{1}{1-2\delta}x-\frac{\delta}{1-2\delta}   & \, \,   \text{if $\delta\le x<1-\delta$}\\
                    1 & \, \,  \text{$1-\delta\le x<1$}\\
                    2sx-2s+1 & \, \,  \text{$1\le x<1+\delta$}\\
                    sx-s+1+s\delta & \, \, \text{$x\ge1+\delta $}\\
                \end{cases}
        \end{equation*}
\endminipage\hfill
\minipage{0.4\textwidth}%
\centering
\includegraphics[width=\linewidth]{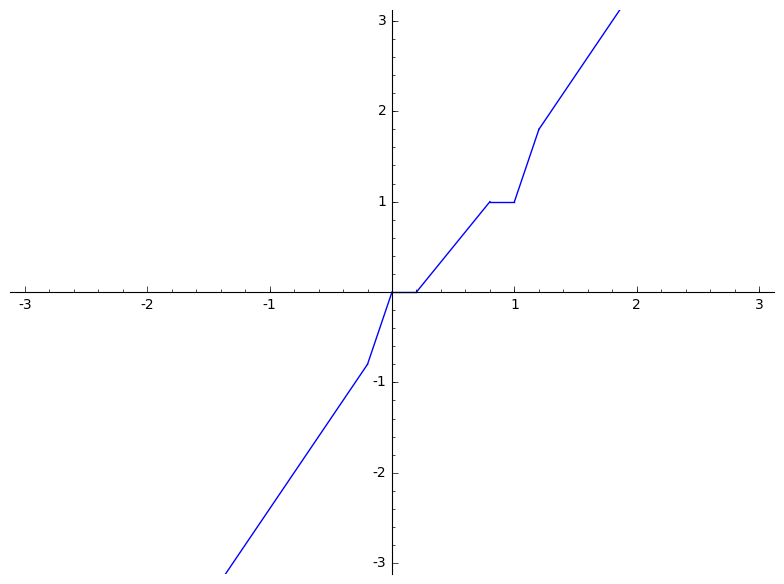}
\captionof{figure}{$\phi_{s,\delta}$ for $s=\frac{1}{5}$ and $\delta=2$ }
\label{fig2}
\endminipage

$\phi_{s,\delta}$ is a continuous piecewise linear function, which has breakpoints: $-\delta, 0, \delta, 1-\delta,1,1+\delta$ and slope values: $s,2s,0,\frac{1}{1-s\delta},0,2s,s$ in each affine piece. \autoref{fig2} is the graph of one $\phi_{s,\delta}$ function constructed by the Python function \sage{phi\textunderscore s\textunderscore delta}.

Let $E_{\delta}=\{(x,y)\in\mathbb{R}^2: -\delta< x< \delta \,\,\,\text{or} \,\,-\delta< y< \delta \,\,\,\text{or} \,\,1-\delta< x+y< 1+\delta\}$. We claim that $\phi_{s,\delta}$ is a continuous restricted maximal gDFF and $\nabla\phi_{s,\delta}(x,y)\ge \delta$ for $(x,y)\notin E_{\delta}$, if $s>1$ and $0<\delta<\min\{\frac{s-1}{2s},\frac{1}{3}\}$. Verifying the above properties of $\phi_{s,\delta}$ is a routine computation by analyzing the superadditivity slack at every vertex in the two-dimensional polyhedral complex of $\phi_{s,\delta}$, which can  also be verified by using metaprogramming\footnote{Interested readers are referred to the function \sage{is\textunderscore superadditive\textunderscore almost\textunderscore strict} in order to check the claimed properties of $\phi_{s,\delta}$.}
 \cite{cutgeneratingfunctionology:online} in SageMath. (More information on $\phi_{s,\delta}$ is provided in Appendix \ref{s:phi-s-delta}.)

\smallbreak

\begin{lemma}

Let $\phi_{\mathrm{pwl}}$ be a piecewise linear continuous restricted maximal gDFF, then for any $\epsilon>0$, there exists a piecewise linear continuous restricted maximal gDFF $\phi_{\mathrm{loose}}$ satisfying: $(i)$ $\|\phi_{\mathrm{loose}}-\phi_{\mathrm{pwl}}\|_{\infty}<\frac{\epsilon}{3}$; $(ii)$ there exists $\gamma>0$ such that $\nabla\phi_{\mathrm{loose}}(x,y)\ge \gamma$ for $(x,y)$ not in $E_{\delta}$.
\end{lemma}

\begin{proof}
By Proposition \ref{limit-slope}, 
let 
$t=$ 
$\lim_{x\to \infty} \frac{\phi_{\mathrm{pwl}}(x)}{x}$, then $tx-t+1\le \phi_{\mathrm{pwl}}(x)\le tx$. We can assume $t>1$, otherwise $\phi_{\mathrm{pwl}}$ is the identity function and the result is trivial. Choose $s=t$ and  $\delta$ small enough such that $0<\delta<\min\{\frac{s-1}{2s},\frac{1}{3}, \frac{1}{q}\}$, where $q $ is the denominator of breakpoints of $\phi_{\mathrm{pwl}}$ in previous lemma. We know that the limiting slope of maximal gDFF $\phi_{t,\delta}$ is also t and $tx-t+1\le \phi_{t,\delta}(x)\le tx$, which implies $\|\phi_{t,\delta}-\phi_{\mathrm{pwl}}\|_{\infty}\le t-1$. 

Define $\phi_{\mathrm{loose}}=(1-\frac{\epsilon}{3(t-1)})\,\phi_{\mathrm{pwl}}+\frac{\epsilon}{3(t-1)}\,\phi_{t,\delta}$. It is immediate to check $\phi_{\mathrm{loose}}$ is restricted maximal since it is a convex combination of two restricted maximal gDFFs.   $\|\phi_{\mathrm{loose}}-\phi_{\mathrm{pwl}}\|_{\infty}<\frac{\epsilon}{3}$ is due to $\|\phi_{t,\delta}-\phi_{\mathrm{pwl}}\|_{\infty}\le t-1$. Based on the property of $\phi_{t,\delta}$, $\nabla\phi_{\mathrm{loose}}(x,y)=(1-\frac{\epsilon}{3(t-1)})\nabla\phi_{\mathrm{pwl}}(x,y)+\frac{\epsilon}{3(t-1)}\nabla\phi_{t,\delta}(x,y)\ge \frac{\epsilon}{3(t-1)}\nabla\phi_{t,\delta}(x,y)\ge \gamma=\frac{\epsilon\delta}{3(t-1)}$ for $(x,y)$ not in $E_{\delta}$.

\end{proof}

\begin{lemma}
Given a piecewise linear continuous restricted maximal gDFF $\phi_{\mathrm{loose}}$ satisfying properties in previous lemma, there exists an extreme gDFF $\phi_{\mathrm{ext}}$ such that $\|\phi_{\mathrm{loose}}-\phi_{\mathrm{ext}}\|_{\infty}<\frac{\epsilon}{3}$.
\end{lemma}

\begin{proof}
Let $s^+\ge 0$ be the largest slope of $\phi_{\mathrm{loose}}$ and $\phi_{\mathrm{loose}}(x)=s^+ x$ for $x\in[-\delta,0]$ where $\delta$ is chosen from previous lemma. Choose $q'\in\mathbb{N}_+$ such that $\frac{1}{q'}s^+<\min\{\frac{\epsilon}{3}, \frac{\gamma}{3}=\frac{\epsilon\delta}{9(t-1)}\}$ and the breakpoints of $\phi_{\mathrm{loose}}$ and $\frac{1}{2}$ are contained in $U=\frac{1}{q'}\mathbb{Z}$. Note that we can always choose a rational $\delta$ to ensure that the last step is feasible. Define a function $g\colon\mathbb{R}\to\mathbb{R}$ and a 2-slope function $\phi_{\mathrm{fill{\text -}in}}\colon[0,1]\to[0,1]$:
\begin{equation*}
            g(x) =
                \begin{cases}
                    0  & \, \,   \text{if $x\ge 0 $} \\
                    s^+x & \, \,   \text{if $x<0$}
                \end{cases}
        \end{equation*}
\begin{equation*}
            \phi_{\mathrm{fill{\text -}in}}(x) =\max_{u\in U}\{\phi_{\mathrm{loose}}(u)+g(x-u)\}
        \end{equation*}
 We claim that $\phi_{\mathrm{fill{\text -}in}}$ is a continuous 2-slope superadditive function and $\phi_{\mathrm{fill{\text -}in}}\le\phi_{\mathrm{loose}}$, $\phi_{\mathrm{fill{\text -}in}}|_{U}=\phi|_U$. The proof is similar to that of \cite[Theorem 3.3]{infinite}. 
 $|\phi_{\mathrm{fill{\text -}in}}(x)-\phi_{\mathrm{loose}}(x)|\le \frac{1}{q}s^+< \frac{\epsilon}{3}$ implies that $\|\phi_{\mathrm{loose}}-\phi_{\mathrm{fill{\text -}in}}\|<\frac{1}{q'}s^+<\frac{\epsilon}{3}$. However, $\phi_{\mathrm{fill{\text -}in}}$ does not necessarily satisfy the symmetry condition. If we symmetrize it and define the following function:
 
 \begin{equation*}
            \phi_{\mathrm{ext}}(x) =
                \begin{cases}
                    \phi_{\mathrm{fill{\text -}in}}(x)  & \, \,   \text{if $0\le x\le \frac{1}{2} $} \\
                    1-\phi_{\mathrm{fill{\text -}in}}(1-x) & \, \,   \text{if $\frac{1}{2}<x\le 1$}
                \end{cases}
        \end{equation*}
        
We claim that $\phi_{\mathrm{ext}}$ is the desired function. It is immediate to check $\phi_{\mathrm{ext}}(0)=0$, $\phi_{\mathrm{ext}}$ is a 2-slope continuous function and it automatically satisfies the symmetry condition. Since we use slope $0$ and $s^+$ to do the fill-in procedure, the limiting slope of $\phi_{\mathrm{ext}}$ at $0^+$ is $0$. $\|\phi_{\mathrm{loose}}-\phi_{\mathrm{ext}}\|_{\infty}<\frac{1}{q'}s^+<\frac{\epsilon\delta}{9(t-1)}$ because they both satisfy the symmetry condition. So we only need to prove $\phi_{\mathrm{ext}}$ is superadditive.

\emph{Case 1:} if $(x,y)$ is not in $E_{\delta}$, $\nabla\phi_{\mathrm{ext}}(x,y)\ge \nabla\phi_{\mathrm{loose}}(x,y)-\frac{\epsilon\delta}{9(t-1)}-\frac{\epsilon\delta}{9(t-1)}-\frac{\epsilon\delta}{9(t-1)}\ge \frac{\epsilon\delta}{3(t-1)}-\frac{\epsilon\delta}{3(t-1)}=0$.

\emph{Case 2:} if $0\le x\le \delta$, there are also three sub cases: 

$(i)$ if $y,x+y\le \frac{1}{2}$, then $\nabla\phi_{\mathrm{ext}}(x,y)=\nabla\phi_{\mathrm{fill{\text -}in}}(x,y)\ge 0$.

$(ii)$ if $y\le \frac{1}{2}$ and $x+y>\frac{1}{2}$, then $\nabla\phi_{\mathrm{ext}}(x,y)=1-\phi_{\mathrm{fill{\text -}in}}(1-x-y)-\phi_{\mathrm{fill{\text -}in}}(x)-\phi_{\mathrm{fill{\text -}in}}(y)\ge 1-\phi_{\mathrm{loose}}(1-x-y)-\phi_{\mathrm{loose}}(x)-\phi_{\mathrm{loose}}(y)\ge 0$.  Here we use the fact that $\phi_{\mathrm{loose}}\ge\phi_{\mathrm{fill{\text -}in}}$ and $\phi_{\mathrm{loose}}$ is a maximal gDFF.

$(iii)$ if $y,x+y>\frac{1}{2}$, then $\nabla\phi_{\mathrm{ext}}(x,y)=(1-\phi_{\mathrm{fill{\text -}in}}(1-x-y))-\phi_{\mathrm{fill{\text -}in}}(x)-(1-\phi_{\mathrm{fill{\text -}in}}(1-y))=\phi_{\mathrm{fill{\text -}in}}(1-y)-\phi_{\mathrm{fill{\text -}in}}(1-x-y)-\phi_{\mathrm{fill{\text -}in}}(x)\ge 0$ due to superadditivity of $\phi_{\mathrm{fill{\text -}in}}$.

\emph{Case 3:} if $0> x\ge -\delta$, based on the choice of $\delta$ and $s^+$, we know $\phi_{\mathrm{ext}}(x)=s^+ x$ for $0> x\ge -\delta$. For any $y\in \mathbb{R}$, $\phi_{\mathrm{ext}}(x+y)-\phi_{\mathrm{ext}}(y)\ge s^+ x=\phi_{\mathrm{ext}}(x)$ since $\phi_{\mathrm{ext}}$ is a 2-slope function and $s^+$ is the larger slope. 

Similarly we can prove $\nabla\phi_{\mathrm{ext}}(x,y)\ge 0$ if $-\delta\le y\le \delta$.

\emph{Case 4:} if $1-\delta\le x+y\le 1+\delta$, let $\beta=1-x-y$ and $-\delta\le \beta\le \delta$, so by case 2 and 3 $\phi_{\mathrm{ext}}(\beta)+\phi_{\mathrm{ext}}(x)\le \phi_{\mathrm{ext}}(\beta+x)$. Then we have  $\phi_{\mathrm{ext}}(x+y)=\phi_{\mathrm{ext}}(1-\beta)=1-\phi_{\mathrm{ext}}(\beta)=1-\phi_{\mathrm{ext}}(\beta)+\phi_{\mathrm{ext}}(x)-\phi_{\mathrm{ext}}(x)\ge 1-\phi_{\mathrm{ext}}(\beta+x)+\phi_{\mathrm{ext}}(x)=1-\phi_{\mathrm{ext}}(1-y)+\phi_{\mathrm{ext}}(x)=\phi_{\mathrm{ext}}(y)+\phi_{\mathrm{ext}}(x)$.

We have shown that $\phi_{\mathrm{ext}}$ is superadditive, then it is a continuous 2-slope strongly maximal gDFF. By the 2-slope theorem (Theorem \ref{t:2-slope}), $\phi_{\mathrm{ext}}$ is extreme.

\end{proof}

Combine the previous lemmas, then we can conclude the main theorem. 

\clearpage

\bibliographystyle{splncs03}
\bibliography{../../../bib/MLFCB_bib}

\clearpage

\appendix
\section*{Appendix}

\section{The proof of \autoref{existence}}
\label{s:proof-existence}
\begin{proof}
\noindent\emph{Part (i)}. 
If the gDFF $\phi$ is already maximal, then it is dominated by itself. We assume $\phi$ is not maximal. Define a set $A=\{\text{valid gDFF $f$:} \,f\ge \phi \}$, which is a partially ordered set. Consider a chain $(\phi_n)_{n=1}^{\infty}$ with $\phi_n\le \phi_{n+1}$ and a function $\phi'(x)=\lim_{n\to \infty} \phi_n(x)$. We claim $\phi'$ is an upper bound of the chain and it is contained in $A$.

First we prove $\phi'$ is a well-defined function. For any fixed $x_0\in\mathbb{R}$, based on the definition of gDFF, we know that $\phi_n(x_0)+\phi(x_0)\le \phi_n(x_0)+\phi_n(x_0)\le 0$. $\phi(x_0)$ is a fixed constant and it forces $\lim_{n\to \infty} \phi_n(x_0)<\infty$. So we know $\phi'(x)=\lim_{n\to \infty} \phi_n(x)<\infty$ for any $x\in\mathbb{R}$.

Next, we prove $\phi'$ is a valid gDFF and dominates $\phi$. It is clear that $\phi'\ge \phi$, so we only need to show $\phi'$ is a valid gDFF. Suppose on the contrary $\phi'$ is not valid, then there exist $(x_i)_{i=1}^{m}$ such that $\sum_{i=1}^{m}x_i\le 1$ and $\sum_{i=1}^{m}\phi'(x_i)=1+\epsilon$ for some $\epsilon>0$. Since there are only finite number of $x_i$, we can choose $n$ large enough such that $\phi'(x_i)<\phi_n(x_i)+\frac{\epsilon}{m}$. Then $1+\epsilon=\sum_{i=1}^{m}\phi'(x_i)<\sum_{i=1}^{m}(\phi_n(x_i)+\frac{\epsilon}{m})\le 1+\epsilon$. The last step is due to the fact that $\phi_n$ is a valid gDFF.

We have shown that every chain in the set $A$ has a upper bound in $A$. By Zorn's lemma, we know there is a maximal element in the set $A$, which is the desired maximal gDFF.

\smallbreak

\noindent\emph{Part (ii)}.
By $(i)$ we only need to show every maximal gDFF $\phi$ is implied via scaling by a restricted maximal gDFF. Based on Remark \ref{rmk}, if $\phi$ is restricted maximal, then it is implied via scaling by itself. If $\phi$ is linear function, then it is implied via scaling by $\phi'(x)=x$.

\smallbreak

\noindent\emph{Part (iii)}.
If $\phi(x)=ax$ for $0\le a\le1$, then it is implied by $\phi'(x)=x$, which is not strongly maximal by definition. We assume $\phi$ is nonlinear.
By $(ii)$ we only need to show every restricted maximal gDFF $\phi$ is implied by a strongly maximal gDFF. If $\phi$ is already strongly maximal, then it is implied by itself. Suppose $\phi$ is not strongly maximal. From the proof of \autoref{strongly-maximal}, we know $1>\lim_{\epsilon\to0^+}\frac{\phi(\epsilon)}{\epsilon}=s>0$.  If the $\limsup$ is not equal to $\liminf$, then we can derive the same contradiction. If $\lim_{\epsilon\to0^+}\frac{\phi(\epsilon)}{\epsilon}=1$, then $\phi$ is the linear functions $\phi(x)=x$. Define a new function $\phi_1(x)=\frac{\phi-sx}{1-s}$ and we want to show $\phi_1$ is a strongly maximal gDFF. Note that $\phi_1(0)=0$, $\phi_1$ is superadditive, $\phi_1(x)+\phi_1(1-x)=1$ and $\lim_{\epsilon\to0^+}\frac{\phi_1(\epsilon)}{\epsilon}=0$. We only need to prove $\phi_1(x)$ is nonnegative if $x$ is nonnegative and near $0$. Suppose on the contrary there exist $x_0>0$ and $\epsilon>0$ such that $\phi(x_0)=sx_0-\epsilon$. There also exists a positive and decreasing sequence $(x_n)^{\infty}_{n=1}$ approaching $0$ and satisfying $\frac{\phi(x_n)}{x_n}>s-\frac{\epsilon}{2x_0}$. Choose $x_n$ small enough and $k\in\mathbb{Z}_{++}$ such that $x_0\ge kx_n\ge x_0-\frac{\epsilon}{2s}$. Since $\phi$ is superadditive and nondecreasing, we have 
$$sx_0-\epsilon=\phi(x_0)\ge \phi(kx_n)\ge k\phi(x_n)> ksx_n-\frac{k\epsilon x_n}{2x_0}\ge sx_0-\frac{\epsilon}{2}-\frac{\epsilon}{2}=sx_0-\epsilon$$ 
The above contradiction implies that $\phi(x)\ge sx$ for positive $x$ near $0$. Therefore $\phi_1$ is strongly maximal and $\phi$ is implied by $\phi_1$.
\end{proof}
 \clearpage

\section{The proof of \autoref{valid-conversion}}
\label{s:proof-conversion}
\begin{theorem}
Given a valid gDFF $\phi$, then the following function is a valid cut-generating function for $Y_{=1}$:
$$\pi_\lambda(x)=\frac{x-(1-\lambda)\,\phi(x)}{\lambda}, \quad 0< \lambda < 1$$
Given a valid cut-generating function $\pi$ for $Y_{=1}$, which is Lipschitz continuous at $x=0$, then there exists $\delta>0$ such that for all $0<\lambda<\delta$ the following function is a valid gDFF:
$$\phi_\lambda(x)=\frac{x-\lambda\,\pi(x)}{1-\lambda}, \quad 0< \lambda < 1$$
\end{theorem}

\begin{proof}
We want to show that $\pi_\lambda$ is a a valid cut-generating function for $Y_{=1}$. Suppose there is a function $y:\mathbb{R} \to \mathbb{Z}_+, \,\text{$y$ has finite support}$, and $\sum_{r\in\mathbb{R}}r\, y(r)= 1$. We want to show that:
\begin{align*}
 & \, \, \sum_{r\in\mathbb{R}} \pi_\lambda(r)\, y(r)\ge 1\quad \text{holds for $\lambda\in (0,1)$}\\
 \Leftrightarrow & \, \,   \sum_{r\in\mathbb{R}} \frac{r-(1-\lambda)\,\phi(r)}{\lambda}\, y(r)\ge 1 \\
  \Leftrightarrow & \, \,   \sum_{r\in\mathbb{R}} (r-(1-\lambda)\,\phi(r))\, y(r)\ge \lambda \\
   \Leftrightarrow & \, \,   \sum_{r\in\mathbb{R}} r\, y(r) - (1-\lambda) \sum_{r\in\mathbb{R}}\phi(r)\,y(r)\ge \lambda \\
   \Leftrightarrow & \, \, \sum_{r\in\mathbb{R}}\phi(r)\,y(r)\le 1
 \end{align*}
 The last step is derived from $\sum_{r\in\mathbb{R}}r\, y(r)= 1$ and $\phi$ is a gDFF.
 
 On the other hand, the Lipschitz continuity of $\pi$ at $0$ guarantees that $\phi_\lambda(x)\ge 0$ for $x\ge 0$ if $\lambda$ is small enough. Then the proof for validity of $\phi_\lambda$ is analogous to the proof above.
\end{proof}




\begin{theorem}
Given a maximal gDFF $\phi$, then the following function is a minimal cut-generating function for $Y_{=1}$:
$$\pi_\lambda(x)=\frac{x-(1-\lambda)\,\phi(x)}{\lambda}, \quad 0< \lambda < 1$$
Given a minimal cut-generating function $\pi$ for $Y_{=1}$, which is Lipschitz continuous at $x=0$, then there exists $\delta>0$ such that for all $0<\lambda<\delta$ the following function is a maximal gDFF:
$$\phi_\lambda(x)=\frac{x-\lambda\,\pi(x)}{1-\lambda}, \quad 0< \lambda < 1$$
\end{theorem}

\begin{proof}
As stated in Theorem \ref{YC:1},  $\pi$ is minimal if and only if $\pi(0)=0$, $\pi$ is subadditive and $\pi(r)=\sup_{k}\{\frac{1}{k}(1-\pi(1-kr)): k\in \mathbb{Z}_+\}$, which is called the generalized symmetry condition. If $\pi_\lambda(x)=\frac{x-(1-\lambda)\,\phi(x)}{\lambda}$, then $\pi_\lambda(0)=0$ and $\pi_\lambda$ is subadditive.
\begin{align*}
& \, \, \sup_{k}\{\frac{1}{k}(1-\pi_\lambda(1-kr)): k\in \mathbb{Z}_+\}\\
= & \, \, \sup_{k}\{\frac{1}{k}(1-\frac{1-kr-(1-\lambda)\,\phi(1-kr)}{\lambda}): k\in \mathbb{Z}_+\} \\
= & \, \, \sup_{k}\{\frac{kr-(1-\lambda)(1-\phi(1-kr))}{k\lambda}: k\in \mathbb{Z}_+\}\\
= & \, \, \sup_{k}\{\frac{r}{\lambda}-\frac{1-\lambda}{\lambda}\frac{1}{k}(1-\phi(1-kr)): k\in \mathbb{Z}_+\}\\
= & \, \, \frac{r}{\lambda}-\frac{1-\lambda}{\lambda}\inf_{k}\{\frac{1}{k}(1-\phi(1-kr)): k\in \mathbb{Z}_+\}\\
= & \, \, \frac{r}{\lambda}-\frac{1-\lambda}{\lambda}\phi(r)\\
= & \, \, \pi_\lambda(r).
\end{align*}
Therefore, $\pi_\lambda$ is minimal.

On the other hand, given a minimal cut-generating function $\pi$, let $\phi_\lambda(x)=\frac{x-\lambda\,\phi(x)}{1-\lambda}$, then it is easy to see the superadditivity and $\phi_\lambda(0)=0$. The generalized symmetry can be proven similarly. The Lipschitz continuity of $\pi$ at $0$ implies that  $\phi_\lambda(x)\ge 0$ for any $x\ge 0$ if $\lambda$ is chosen properly.
\end{proof}

\begin{theorem}
Given a restricted maximal gDFF $\phi$, then the following function is a restricted minimal cut-generating function for $Y_{=1}$:
$$\pi_\lambda(x)=\frac{x-(1-\lambda)\,\phi(x)}{\lambda}, \quad 0< \lambda < 1$$
Given a restricted minimal cut-generating function $\pi$ for $Y_{=1}$, which is Lipschitz continuous at $x=0$, then there exists $\delta>0$ such that for all $0<\lambda<\delta$ the following function is a restricted maximal gDFF:
$$\phi_\lambda(x)=\frac{x-\lambda\,\pi(x)}{1-\lambda}, \quad 0< \lambda < 1$$
\end{theorem}

\begin{proof}
As stated in Theorem \ref{YC:2}, $\pi$ is restricted minimal if and only if $\pi(0)=0$, $\pi$ is subadditive and $\pi(r)=\sup_{k}\{\frac{1}{k}(1-\pi(1-kr)): k\in \mathbb{Z}_+\}$, and $\pi(1)=1$. Given a restricted maximal gDFF $\phi$, we have $\phi(1)=1$, which implies $\pi_\lambda(1)=1$.

On the other hand, a restricted minimal $\pi$ satisfying $\pi(1)=1$, then $\phi_\lambda(1)=1$. Based on the maximality of $\phi_\lambda$, we know $\phi_\lambda $ is restricted maximal.
\end{proof}

\clearpage

\section{A complete proof of \autoref{lemma2}}
\label{s:proof-lemma2}

\begin{proof}
From Lemma \ref{lemma1} we
know $\phi$ is continuous at $0$ from the right. Suppose 
$\phi(x)=sx$, $x\in[0,x_1)$ and $s>0$. 

$\phi$ is not strictly increasing if $s=0$. In order to satisfy the superadditivity, $s$ should be the smallest slope value. $s\le1$ since $\phi(1)\le1$ and $\phi$ is  nondecreasing, which means even if $\phi$ is discontinuous, $\phi$ can only jump up at discontinuities. Similarly if $s=1$, then $\phi(x)=x$.

Next, we can assume $0<s<1$. Define a function: $$\phi_1(x)=\frac{\phi(x)-sx}{1-s}$$
Clearly $\phi_1(x)=0$ for $x\in [0,x_1)$. $\phi_1$ is superadditive because it is obtained by subtracting a linear function from a superadditive function. 

\begin{align*}
  \phi_1(r) & \, \,=\frac{\phi(r)-sr}{1-s}\\
  & \, \, = \frac{1}{1-s}[\inf_{k}\{\frac{1}{k}(1-\phi(1-kr)): k\in \mathbb{Z}_+\} -sr
]\\
 & \, \,=  \frac{1}{1-s}[\inf_{k}\{\frac{1}{k}(1-[(1-s)\phi_1(1-kr)+s(1-kr)]): k\in \mathbb{Z}_+\} -sr]\\
 & \, \,= \frac{1}{1-s}[\inf_{k}\{\frac{1}{k}[(1-s)+skr-(1-s)\phi_1(1-kr)]: k\in \mathbb{Z}_+\} -sr]\\
 & \, \,= \frac{1}{1-s}\inf_{k}\{\frac{1}{k}[(1-s)-(1-s)\phi_1(1-kr)]: k\in \mathbb{Z}_+\} \\
 & \, \, = \inf_{k}\{\frac{1}{k}(1-\phi_1(1-kr)): k\in \mathbb{Z}_+\} 
\end{align*}

The above equation shows that $\phi_1$ satisfies the generalized symmetry condition. Therefore, $\phi_1$ is also a maximal gDFF. $\phi(x)=sx+(1-s)\phi_1(x)$ implies $\phi$ is not extreme, since it can be expressed as a convex combination of two different maximal gDFFs: $x$ and $\phi_1$.
\end{proof}

\clearpage

\section{Two-dimensional polyhedral complex}
\label{s:phi-s-delta}
In this section, we explain the reason why we can only check the superadditive slack at finitely many vertices in the two-dimensional polyhedral complex, in order to prove $\nabla\phi_{s,\delta}(x,y)\ge \delta$ for $(x,y)\notin E_{\delta}$, if $s>1$ and $0<\delta<\min\{\frac{s-1}{2s},\frac{1}{3}\}$.

Let $\phi\colon \mathbb{R}\to\mathbb{R}$ be a piecewise linear function with finitely many pieces with breakpoints $x_1<x_2<\dots<x_n$. 
To express the domains of
linearity of $\nabla\phi(x,y)$, and thus domains of additivity and strict
superadditivity, we introduce the two-dimensional polyhedral complex
$\Delta\P$. 
The faces $F$ of the complex are defined as follows. Let $I, J, K \in
\P$, so each of $I, J, K$ is either a breakpoint of $\phi$ or a closed
interval delimited by two consecutive breakpoints including $\pm\infty$. Then 
$ F = F(I,J,K) = \setcond{\,(x,y) \in \R \times \R}{x \in I,\, y \in J,\, x + y \in
  K\,}$. 
Let $F \in \Delta\P$ and observe that the piecewise linearity of $\phi$ induces piecewise
linearity of $\nabla \phi$.

\begin{lemma}
\label{code}
Let $\phi\colon \mathbb{R}\to \mathbb{R}$ be a continuous piecewise linear function with finitely many pieces with breakpoints $x_1<x_2<\dots<0< \dots<x_n$ and $\phi$ has the same slope $s$ on $(-\infty, x_1]$ and $[x_n,\infty)$. Consider a one-dimensional unbounded face $F$ where one of $I,J,K$ is a finite breakpoint and the other two are unbounded closed intervals, $(-\infty, x_1]$ or $[x_n,\infty)$. Then $\nabla\phi(x,y)$ is a constant along the face $F$.
\end{lemma}

\begin{proof}
We only provide the proofs for one
 case, the proofs for other cases are similar.

Suppose $I=\{x_i\}$, $J=K=[x_n,\infty)$. The vertex of $F$ is $(x,y)=(x_i,x_n)$ if $x_i\ge 0$ and $(x,y)=(x_i,x_n-x_i)$ if $x_i< 0$. If $x_i\ge 0$, we claim that $\nabla\phi(x,y)=\nabla\phi(x_i,x_n)$ for $(x,y)\in F$.
$$\nabla\phi(x,y)=\phi(x_i+y)-\phi(x_i)-\phi(y)=\phi(x_i+x_n)-\phi(x_i)-\phi(x_n)=\nabla\phi(x_i,x_n)$$
The second step in the above equation is due to $\phi$ is affine on $[x_n,\infty)$ and $x_i+x_n,y\ge x_n$.

 If $x_i< 0$, we claim that $\nabla\phi(x,y)=\nabla\phi(x_i,x_n-x_i)$ for $(x,y)\in F$.
 \begin{align*}
 \nabla\phi(x,y)& \, \,=\phi(x_i+y)-\phi(x_i)-\phi(y)\\
  & \, \, =(\phi(x_n)+s(x_i+y-x_n))-\phi(x_i)-(\phi(x_n-x_i)+s(x_i+y-x_n))\\
  & \, \, = \phi(x_n)-\phi(x_i)-\phi(x_n-x_i)=\nabla\phi(x_i,x_n-x_i)
\end{align*}

The second step in the above equation is due to $\phi$ has slope $s$ on $[x_n,\infty)$ and $x_n-x_i, x_i+y\ge x_n$.

\smallbreak

\emph{Case 2:} Suppose $K=\{x_i\}$, $I=[x_n,\infty)$ and $J=(-\infty,x_1]$. The vertex of $F$ is $(x,y)=(x_n,x_i-x_n)$ if $x_i\le x_1+x_n$ and $(x,y)=(x_i-x_1,x_1)$ if $x_i> x_1+x_n$.

If $x_i\le x_1+x_n$, we claim that $\nabla\phi(x,y)=\nabla\phi(x_n,x_i-x_n)$ for $(x,y)\in F$.
\begin{align*}
 \nabla\phi(x,y)& \, \,=\phi(x_i)-\phi(x)-\phi(x_i-x)\\
  & \, \, =\phi(x_i)-(\phi(x_n)+s(x-x_n))-(\phi(x_i-x_n)-s(x-x_n))\\
  & \, \, = \phi(x_i)-\phi(x_n)-\phi(x_i-x_n)=\nabla\phi(x_n,x_i-x_n)
\end{align*}

The second step in the above equation is due to $\phi$ has the same slope $s$ on $(-\infty, x_1]$ and $[x_n,\infty)$ and $x\ge x_n$, $y=x_i-x\le x_i-x_n\le x_1$.

If $x_i> x_1+x_n$, we claim that $\nabla\phi(x,y)=\nabla\phi(x_i-x_1,x_1)$ for $(x,y)\in F$.
\begin{align*}
 \nabla\phi(x,y)& \, \,=\phi(x_i)-\phi(x_i-y)-\phi(y)\\
  & \, \, =\phi(x_i)-(\phi(x_i-x_1)+s(x_1-y))-(\phi(x_1)-s(x_1-y))\\
  & \, \, = \phi(x_i)-\phi(x_i-x_1)-\phi(x_1)=\nabla\phi(x_i-x_1,x_1)
\end{align*}

The second step in the above equation is due to $\phi$ has the same slope $s$ on $(-\infty, x_1]$ and $[x_n,\infty)$ and $y\ge x_1$, $x=x_i-y\ge x_i-x_1\ge x_n$.

Therefore,  $\nabla\phi(x,y)$ is a constant for $(x,y)$ in any fixed one-dimensional unbounded face.

\end{proof}

By using the piecewise linearity of $\nabla\phi$, we can prove the following lemma.

\begin{lemma}
Define the two-dimensional polyhedral complex
$\Delta\P$ of the function $\phi_{s,\delta}$. 
If 
$\nabla\phi_{s,\delta}(x,y)\ge \delta$ for any zero-dimensional face $(x,y)\notin E_{\delta}$, then $\nabla\phi_{s,\delta}(x,y)\ge \delta$ for $(x,y)\notin E_{\delta}$.
\end{lemma}

\begin{proof}
Observe that $\mathbb{R}^2-E_{\delta}$ is the union of finite two-dimensional faces. So we only need to show $\nabla\phi_{s,\delta}(x,y)\ge \delta$ for $(x,y)\notin E_{\delta}$ and $(x,y)$ in some two-dimensional face $F$. 

If $F$ is  bounded, then $\nabla\phi_{s,\delta}(x,y)\ge \delta$ since the inequality holds for vertices of $F$ and $\nabla\phi$ is affine over $F$. 

If $F$ is unbounded and suppose it is enclosed by some bounded one-dimensional faces and unbounded one-dimensional faces. For those bounded one-dimensional faces, $\nabla\phi_{s,\delta}(x,y)\ge \delta$ holds since the inequality holds for vertices. For any unbounded one-dimensional face $F'$, by Lemma \ref{code}, the $\nabla\phi$ is constant and equals to the value at the vertex of $F'$. We have showed that $\nabla\phi_{s,\delta}(x,y)\ge \delta$ holds for any $(x,y)$ in the enclosing one-dimensional faces, then the inequality holds for $(x,y)\in F$ due to the piecewise linearity of $\nabla\phi$.
\end{proof}

\begin{remark}
The code is available at \cite{cutgeneratingfunctionology:online}: \\https://github.com/mkoeppe/cutgeneratingfunctionology

In the code, we define a parametric family of functions $\phi_{s,\delta}$ with two variable $s$ and $\delta$. It is clear that $\phi_{s,\delta}$ satisfies the symmetry condition. Although $\phi_{s,\delta}$ is defined in the unbounded domain $\mathbb{R}$, the $\nabla\phi$ only depends on the values at the vertices of $\P$ which is a bounded and finite set. Therefore, it suffices to check  $\nabla\phi_{s,\delta}\ge \delta$ at those finitely many vertices.
\end{remark}

\end{document}